\documentclass[11pt]{amsart}
\usepackage{amscd}
\usepackage{amssymb}
\usepackage[margin=1in,includeheadfoot]{geometry}
\usepackage{cite}
\usepackage [latin1]{inputenc}
\usepackage{tabularx,booktabs,tikz}
\usepackage{caption}
\usepackage{amsmath}
\usepackage{amsfonts}

\usepackage{amscd}
\usepackage{amsthm}
\usepackage{amssymb} \usepackage{latexsym}
\usepackage{eufrak}
\usepackage{euscript}
\usepackage{epsfig}
\usepackage{graphics}
\usepackage{array}
\usepackage{enumerate}
\usepackage{dsfont}
\usepackage{color}
\usepackage{wasysym}
\usepackage{hyperref}
\usepackage{pdfsync}

\def\beq{\begin{equation}}
\def\eeq{\end{equation}}
\def\ba{\begin{array}}
\def\ea{\end{array}}

\def\R{\mathbb R}

\newtheorem{thm}{Theorem}[section]
\newtheorem{lm}[thm]{Lemma}

\newtheorem{crl}[thm]{Corollary}

\theoremstyle{definition}
\newtheorem{rem}[thm]{Remark}

\theoremstyle{remark}

\begin{document}
\pagestyle{plain}
\title{A class of semilinear elliptic equations on lattice graphs }
\author{Bobo Hua}
\email{bobohua@fudan.edu.cn}
\address{Bobo Hua: School of Mathematical Sciences, LMNS, Fudan University, Shanghai 200433, China; Shanghai Center for Mathematical Sciences, Fudan University, Shanghai 200433, China}

\author{Ruowei Li}
\email{rwli19@fudan.edu.cn}
\address{Ruowei Li: School of Mathematical Sciences, Fudan University, Shanghai 200433, China}

\author{Lidan Wang}
\email{wanglidan@fudan.edu.cn}
\address{Lidan Wang: School of Mathematical Sciences, Fudan University, Shanghai 200433, China}


 \thanks{
The corresponding author is Lidan Wang, wanglidan@fudan.edu.cn.
}

\begin{abstract}
In this paper, we study the semilinear elliptic equation of the form
\begin{eqnarray*}
-\Delta u+a(x)|u|^{p-2}u-b(x)|u|^{q-2}u=0
\end{eqnarray*}
on lattice graphs $\mathbb{Z}^{N}$, where $N\geq 2$ and $2\leq p<q<+\infty$.
By the Br\'{e}zis-Lieb lemma and concentration compactness principle, we prove the existence of positive solutions to the above equation with constant coefficients $\bar{a},\bar{b}$ and the decomposition of bounded Palais-Smale sequences for the functional with variable coefficients, which tend to some constants $\bar{a},\bar{b}$ at infinity, respectively.
\end{abstract}

 \maketitle



\section{Introduction}
The semilinear elliptic equations on Euclidean spaces and Riemannian manifolds have been extensively studied in the literature. See for examples
\cite{CL3,CL1,DN,GS1,GS2,L1,L2,JT,KP,Y1,ZC} for various types of semilinear elliptic equations on Euclidean spaces. In Particular, they are the solitary waves of evolution equations \cite{BL3,B1,P,S2,W}. For the equations on Riemannian manifolds, such as Yamabe type equations and Kazdan-Warner equations, we refer readers to \cite{A,CL2,CL4,DL,KW,MRS,Y}.

Recently, people began to consider semilinear equations on discrete spaces. For example, a class of semilinear equations with
the nonlinearity of power type, including the well known Yamabe type equations, have been studied on graphs, see \cite{CZ,G,GJ,GJ1,GLY2,GLY3, HSZ,HL,ZL,ZL0,ZZ}. A class of semilinear
equations with the exponential nonlinearity, so-called Kazdan-Warner equations and the Liouville equations,
also have been studied in these papers \cite{G0,GJ3,GJH,GLY1,KS,ZC1} on graphs.

In this paper, we consider natural discrete models for Euclidean spaces, the integer lattice graphs. The $N$-dimensional integer lattice graph, denoted by $\mathbb{Z}^N$, consists of the set of vertices $\mathbb{V}=\mathbb{Z}^N$ and the set of edges $\mathbb{E}=\{(x,y): x,\,y\in\mathbb{Z}^{N},\,\underset {{i=1}}{\overset{N}{\sum}}|x_{i}-y_{i}|=1\}.$

We study the semilinear elliptic equation
\begin{equation}\label{1-0}
-\Delta u+a(x)|u|^{p-2}u-b(x)|u|^{q-2}u=0
\end{equation}
on lattice graphs $\mathbb{Z}^N$, where $N\geq 2$ and $2\leq p<q<+\infty.$ In this paper, we always assume that the coefficients satisfy
\begin{itemize}
\item[(A1)]  $a(x):\mathbb{Z}^N\rightarrow \R$ is a nonnegative function satisfying
\begin{eqnarray*}\label{1-1}
\lim_{|x|\to +\infty}a(x)=\bar{a}>0;
\end{eqnarray*}
\item[(A2)] $b(x):\mathbb{Z}^N\rightarrow \R$ is a nonnegative function satisfying
\begin{eqnarray*}\label{1-2}
\lim_{|x|\to +\infty}b(x)=\bar{b}>0.
\end{eqnarray*}
\end{itemize}

The equation (\ref{1-0}) has been studied by Kuzin and Pohozaev \cite{KP} on $\R^N.$ The first key step lies in the analysis of its limiting equation
 \begin{equation}\label{1-3}
 -\Delta u+\bar{a}|u|^{p-2}u-\bar{b}|u|^{q-2}u=0.
 \end{equation}
More precisely, Kuzin and Pohozaev first proved the existence of positive ground state solutions to the equation (\ref{1-3}) on $\R^N.$ Then, they studied the behavior of Palais-Smale sequences for the functional corresponding to the equation (\ref{1-0}) on $\R^N.$  As an application, they proved the existence of nontrivial solutions to the equation (\ref{1-0}) with variable coefficients by using the previous results.

We are concerned with the existence of positive solutions to the equation (\ref{1-3}) and the behavior of Palais-Smale sequences for the functional related to the equation (\ref{1-0}) on lattice graphs $\mathbb{Z}^{N}$. As we know, during the past years, there are a lot of existence results for the semilinear equation
 \begin{equation}\label{a}
-\Delta u=g(u)
\end{equation} on Euclidean spaces. For example, Berestycki and Lions \cite{BL1} studied the equation (\ref{a}) on $\R^N$ with some assumptions on $g$
 and proved the existence of a positive, radially symmetric, least-energy solution $u$ by means of a constrained minimization method and the Strauss lemma \cite{S2}. The Strauss lemma is crucial for the argument which brings the compactness for radial $H^1$ functions, while the compactness fails for general $H^1$ functions on $\R^N.$  For more related works, we refer to \cite{BL2,BL3,B1,BLP,HIT,M,M1,P,S2}.

 Now we consider the equation \eqref{1-2}. In the continuous setting, Kuzin and Pohozaev \cite{KP} proved the existence of positive solutions to the equation (\ref{1-3}) via the variational method restricted on radial functions. We sketch their proof briefly. Let $\mathcal{E}_{p,q}$ be the completion of $C_c^\infty(\R^N),\, N\geq 3,$ with respect to the norm
\begin{eqnarray*}
\|u\|_{\mathcal{E}_{p,q}}=\|\nabla u\|_{2}+\|u\|_{p}+\|u\|_{q},\qquad 1\leq p\leq q<+\infty.
\end{eqnarray*}
Write $\mathcal{E}_{p}=\mathcal{E}_{p,p}$ and $\mathcal{E}^{rad}_{p,q}$ for the space of radial functions in $\mathcal{E}_{p,q}.$ For any $2\leq p<q<2^*:=\frac{2N}{N-2},$  define two functionals on $\mathcal{E}^{rad}_{p,q},$
 $$\widetilde{J}_1(u)=\frac{1}{2}\|\nabla u\|_{2}^{2},\quad \widetilde{J}_{2}(u)=\frac{\alpha}{p}\|u\|_{p}^{p}-\beta\|u\|_{q}^{q},\quad \alpha,\,\beta>0.$$ Then consider the variational problem
 $$\inf\{\widetilde{J}_1(u):u\in \mathcal{E}^{rad}_{p,q},\widetilde{J}_2(u)=-1\}.$$ By the compactness $\mathcal{E}^{rad}_{p} \looparrowright L^{rad}_{q}$ for $p<q<2^*,$ one can prove the existence result by the standard argument.

 In our setting, the main difficulty for the analysis is that there is no proper counterparts for radial functions on $\mathbb{Z}^N.$ Hence the Strauss lemma on $\mathbb{Z}^N$ is unknown, and we do not have the compactness in this problem. Moreover, the rescaling trick on the Euclidean space is not available on the lattice graph. For notions on lattice graphs, we refer to Section~2. To circumvent the difficulties, we define functionals $J_1, J_2 : \mathcal{E}_{p,q}\rightarrow\R$ in the form
\begin{eqnarray*}
J_{1}(u)=\frac{1}{2}\|\nabla u\|_{2}^{2}+\frac{\alpha}{p}\|u\|_{p}^{p},\qquad
J_{2}(u)=\beta\|u\|_{q}^{q}
\end{eqnarray*}
with $\alpha,\,\beta>0,\,2\leq p<q<+\infty$ and consider the variational problem $$\inf\{J_1(u):u\in \mathcal{E}_{p,q},J_2(u)=1\}.$$
Note that the above variational problem does not work for the continuous case, but we can prove the existence of a minimizer in the discrete setting. The key point of our proof is that any minimizing sequence, passing to a subsequence and with proper translations, has a non-vanishing limit. Then we can prove the result by analyzing the bubbling structure of the limit. Namely, we prove that the minimizing sequence $\{u_n\}$ is decomposed into infinitely many bubbles
$$u_n(x)=u_{(0)}(x)+\underset {i=1}{\overset{\infty}{\Sigma}}u_{(i)}(x-x^{i}_n),$$
where $|x^{i}_n|\rightarrow+\infty$ and $|x^{j}_n-x^{i}_n|\rightarrow+\infty,\,j\neq i$ as $n\rightarrow+\infty$. Then, the result follows from the convexity of the functionals. Now we state our first main result of this paper.
\begin{thm}\label{t-9}
{\rm
Let $N\geq 2$ and $2\leq p<q<+\infty.$ Then for any positive constant $\tilde{a}$, there are infinitely many positive constants $\tilde{b}$ such that the semilinear elliptic equation
\begin{equation}\label{0-0}
-\Delta u+\tilde{a}|u|^{p-2}u-\tilde{b}|u|^{q-2}u=0
\end{equation}
has at least a positive solution $u\in\mathcal{E}_{p,q}$ on lattice graphs $\mathbb{Z}^{N}$.

}
\end{thm}

\begin{rem}\label{4}
{\rm

\begin{itemize}

\item[(i)] In the continuous setting, the powers of nonlinear terms satisfy $p<q<2^*$. However, we remove the constraint that $q<2^*$ (subcritical) since we have the embedding $l^{s}$ into $l^{t}$ for $s<t$ in the discrete setting.

\item[(ii)] In the discrete setting, Ge, Zhang and Lin \cite{G,ZL0} proved the $p$-th Yamabe equation has a positive solution for one pair of coefficients on finite graphs. While we can prove that for any $\tilde{a}>0$, there are infinitely many $\tilde{b}>0$ such that
    the equation (\ref{0-0}) has at least a positive solution. Since the rescaling trick is not available on the lattice graph $\mathbb{Z}^N$, we can not prove that the equation (\ref{0-0}) has a positive solution for any pair $(\tilde{a},\tilde{b})\in (0,\infty)\times(0,\infty)$ as in the continuous case.

\item[(iii)] In the discrete setting, the authors in \cite{GJ,GLY3,ZZ} proved the existence of nontrivial solutions to Yamabe equations and Schr\"{o}dinger equations with the coefficient $a(x)$ satisfying $a(x)\rightarrow+\infty$ as $|x|\rightarrow+\infty$, which guarantees the compact embedding  on $\mathbb{Z}^N$. Since our assumption on $a(x)$ is weaker, particularly for the constant coefficient $\tilde{a}>0$, we have to use the Br\'{e}zis-Lieb lemma and concentration compactness principle to prove the existence of positive solution to the equation (\ref{0-0}). This differs from the continuous case. In the continuous case, Kuzin and Pohozaev \cite{KP} proved the existence result via the variational method restricted on radial functions since the Strauss lemma brings the compactness for radial functions.
\item[(iv)] We prove the existence of positive solutions to the equation (\ref{0-0}), while we do not know whether they are in fact ground state solutions. The statement that they are ground state solutions follows from the Pohozaev identity and the scaling trick on Euclidean spaces $\R^N,$ see \cite[Theorem~15.2]{KP}, which are unknown on lattice graphs $\mathbb{Z}^N.$ This leads to an open problem for the existence of nontrivial solutions to the equation (\ref{1-0}) with variable coefficients on $\mathbb{Z}^N$.

    \end{itemize}
}
\end{rem}

 Next, we analyze the behavior of Palais-Smale sequences for the functional related to the equation (\ref{1-0}) on lattice graphs $\mathbb{Z}^N$. In the continuous setting, Kuzin and Pohozaev obtained the decomposition of Palais-Smale sequences by using the Br\'{e}zis-Lieb lemma and concentration compactness method. We follow their ideas to prove the result on $\mathbb{Z}^N$. Namely, we set up a discrete version of the Br\'{e}zis-Lieb lemma and concentration compactness principle and adopt them to get the desired result.

We firstly introduce the functionals corresponding to the equations (\ref{1-0}) and (\ref{1-2}), respectively. The energy functional $\Phi: \mathcal{E}_{p,q}\rightarrow\R$ associated to the equation (\ref{1-0}) is given by
\begin{eqnarray*}
\Phi(u)=\frac{1}{2}\int_{\mathbb{Z}^N}|\nabla u|^{2}\,d\mu+\frac{1}{p}\int_{\mathbb{Z}^N}a(x)|u|^{p}\,d\mu-\frac{1}{q}\int_{\mathbb{Z}^N}b(x)|u|^{q}\,d\mu,
\end{eqnarray*} where $|\nabla u|$ is the length of discrete gradient of $u,$ and $\mu$ is the counting measure on $\mathbb{Z}^N.$
Formally, our problem has a variational structure. Let $C_c(\mathbb{Z}^N)$
be the set of all functions with finite support on $\mathbb{Z}^N$. Then given $u\in \mathcal{E}_{p,q}$ and $\phi\in C_{c}(\mathbb{Z}^{N})$, the Gateaux derivative of $\Phi$ is given by
\begin{eqnarray*}\label{1-7}
\langle\Phi'(u),\phi\rangle=\int_{\mathbb{Z}^N}\nabla u\nabla \phi \,d\mu+\int_{\mathbb{Z}^N}a(x)|u|^{p-2}u\phi\,d\mu-\int_{\mathbb{Z}^N}b(x)|u|^{q-2}u\phi\,d\mu.
\end{eqnarray*}
The function $u\in \mathcal{E}_{p,q}$ is called a weak solution of (\ref{1-0}), if this Gateaux derivative is zero in every direction $\phi\in C_{c}(\mathbb{Z}^{N})$.

Corresponding to the equation (\ref{1-2}), we define the variational functional $\bar{\Phi}: \mathcal{E}_{p,q}\rightarrow\R$ as
\begin{eqnarray*}
\bar{\Phi}(u)=\frac{1}{2}\int_{\mathbb{Z}^N}|\nabla u|^{2}\,d\mu+\frac{1}{p}\int_{\mathbb{Z}^N}\bar{a}|u|^{p}\,d\mu-\frac{1}{q}\int_{\mathbb{Z}^N}\bar{b}|u|^{q}\,d\mu.
\end{eqnarray*}

Let $I_{0}(\bar{\Phi}):=\underset {u\in \mathcal{P}(\bar{\Phi})}{\inf}\bar{\Phi},$ where $\mathcal{P}(\bar{\Phi})$ is the set of all nontrivial critical points $u\in\mathcal{E}_{p,q}$ of the functional $\bar{\Phi}$.

Recall that, for a given functional $\Phi\in C^1(\mathcal{E}_{p,q},\R)$, we say that $\{u_n\}\subset\mathcal{E}_{p,q}$ is a Palais-Smale sequence of the functional $\Phi$ at level $c\in\R$, if it satisfies that
\begin{eqnarray*}
\Phi(u_n)\rightarrow c, \qquad \text{and}\qquad
\Phi'(u_n)\rightarrow 0, \qquad \text{in~}\mathcal{E}^{*}_{p,q},
\end{eqnarray*}
where $\mathcal{E}^{*}_{p,q}$ is the dual space of $\mathcal{E}_{p,q}$.

Now, we give our second main result, a discrete analog of the decomposition of bounded Palais-Smale sequences of the functional $\Phi$ in \cite[Theorem~14.1]{KP}.
\begin{thm}\label{t-0}
{\rm
Let $N\geq 2$ and $2\leq p<q<+\infty.$ Assume that (A1) and (A2) hold and $I_{0}(\bar{\Phi})>0$. If $\{u_n\}\subset\mathcal{E}_{p,q}$ is a Palais-Smale sequence of the functional $\Phi$ at level $c\in\R$, then there exist nonnegative integer $k$,  a solution $u_{(0)}\in\mathcal{E}_{p,q}$ to the equation (\ref{1-0}), $k$ nontrivial solutions $u_{(i)}\in\mathcal{E}_{p,q}$ to the equation (\ref{1-2}) and $k$ sequences $\{y_{n}^{i}\}\subset\mathbb{Z}^N,\, i=1,2\cdots k$ with $|y_{n}^{i}|\rightarrow+\infty$ and $|y_{n}^{j}-y_{n}^{i}|\rightarrow+\infty$ for $j\neq i$ as $n\rightarrow+\infty$, such that up to a subsequence,
\begin{itemize}
\item[(i)] $(u_n-u_{(0)}-\underset {{i=1}}{\overset{k}{\Sigma}}u_{(i)}(x-y_{n}^{i}))\rightarrow0,\qquad\text{in~}\mathcal{E}_{p,q};$
\item[(ii)] $\Phi(u_{(0)})+\underset {{i=1}}{\overset{k}{\Sigma}}\bar{\Phi}(u_{(i)})=c.$
\end{itemize}
}
\end{thm}

\begin{rem}\label{l}
{\rm
Let $(\Gamma,S)$ be a Cayley graph of a discrete group with a finite generating set $S$. Particularly, the lattice graph $\mathbb{Z}^N$ is a Cayley graph of a free abelian group. By similar arguments as in this paper, the results of Theorem \ref{t-9} and Theorem \ref{t-0} hold on  Cayley graphs.

}
\end{rem}

This paper is organized as follows. In Section 2, we first introduce the setting for graphs, and then give some basic results. In Section 3, we prove some important lemmas for the functionals $\Phi$ and $\bar{\Phi}$. In Section 4, we prove the main results, Theorem \ref{t-9} and \ref{t-0}, by using the Br\'{e}zis-Lieb lemma and concentration compactness principle.

\section{Preliminaries}
 In this section, we first introduce the setting for graphs, and then give some preliminary results.

  Let $G=(\mathbb{V},\mathbb{E})$ be a connected, locally finite graph, where $\mathbb{V}$ denotes the vertex set and $\mathbb{E}$ denotes the edge set. We call vertices $x$ and $y$ neighbors, denoted by $x\sim y$, if there is an edge connecting them, i.e. $(x,y)\in \mathbb{E}$.

For any $x,\,y\in \mathbb{V}$, the distance $|x-y|$ is defined as the minimum number of edges connecting $x$ and $y$. Let $B_{R}(x)=\{y\in \mathbb{V}: |x-y|\leq R\}$ be the ball centered at $x$ with radius $R$ in $\mathbb{V}$. We write $B_{R}=B_{R}(0)$ and $B_{R}^{c}=\mathbb{V}\setminus B_{R}$ for convenience. In this paper, the constant $C$ may change line to line.

We denote the space of functions on $\mathbb{V}$ by $C(\mathbb{V})$.
For $u\in C(\mathbb{V})$, its support set is defined as $\text{supp}(u)=\{  x\in \mathbb{V}: u(x)\neq 0 \}$. Let $C_c(\mathbb{V})$
be the set of all functions with finite support.

For  $u\in C(\mathbb{V})$, we define the difference operator, for any $x\sim y $, as $\nabla_{xy}u=u(y)-u(x).$ The gradient form $\Gamma,$ called the ``carr\'e du champ" operator, is defined as
\begin{eqnarray*}
\Gamma(u,v)(x)=\frac{1}{2}\underset {y\sim x}{\sum}(u(y)-u(x))(v(y)-v(x)):=\nabla u \nabla v.
\end{eqnarray*}
In particular, we write $\Gamma(u)=\Gamma(u,u)$ and denote the length of its gradient by
\begin{eqnarray*}
|\nabla u|(x)=\sqrt{\Gamma(u)(x)}=(\frac{1}{2}\underset {y\sim x}{\sum}(u(y)-u(x))^{2})^{\frac{1}{2}}.
\end{eqnarray*}
For $u\in C(\mathbb{V})$, we define the Laplacian of $u$ as
\begin{eqnarray*}
\Delta u(x)=\underset {y\sim x}{\sum}(u(y)-u(x)).
\end{eqnarray*}
The space $l^p(\mathbb{V})$ is defined as
\begin{eqnarray*}
l^{p}(\mathbb{V})=\{u\in C(\mathbb{V}): \|u\|_p<\infty\},
\end{eqnarray*}
where
\indent
\[\ \|u\|_{p}=\left\{\begin{array}{ll} (\sum\limits_{x\in \mathbb{V}}|u(x)|^p)^{\frac{1}{p}},\quad &\text {if}~1\leq p<\infty, \\ \underset {x\in \mathbb{V}}{\sup}|u(x)|,\quad &\text {if}~p=\infty. \end{array}\right. \]
We can define the space $l^{p}(\mathbb{E})$ analogously. In addition, we define the space $\mathcal{E}_{p,q}$ as the completion of $C_c(\mathbb{V})$ with respect to the norm
\begin{eqnarray*}
\|u\|_{\mathcal{E}_{p,q}}=\|\nabla u\|_{2}+\|u\|_{p}+\|u\|_{q},\qquad 1\leq p\leq q<+\infty.
\end{eqnarray*}
Write $\mathcal{E}_{p}=\mathcal{E}_{p,p}$. Clearly, $\mathcal{E}_2$ is a discrete analog of $H^1$ space. Furthermore, we can define the general space $\mathcal{E}_{p(a),q(b)}$ similarly, that is, the closure of $C_c(\mathbb{V})$ in the norm
\begin{eqnarray*}
\|u\|_{\mathcal{E}_{p(a),q(b)}}=\|\nabla u\|_{2}+\|a^{\frac{1}{p}}u\|_{p}+\|b^{\frac{1}{q}}u\|_{q},\qquad 1\leq p\leq q<+\infty,
\end{eqnarray*}
where $a$ and $b$ are called weighted functions.  Note that the spaces defined above are Banach spaces. In particular, they are reflexive spaces if $p>1$.

For convenience, for any $u\in C(\mathbb{V})$, we always write
$
\int_{\mathbb{V}}u\,d\mu:=\sum\limits_{x\in \mathbb{V}}u(x),
$
where $\mu$ denotes the counting measure in $\mathbb{V}$, that is, for any subset $A\subset \mathbb{V},\,\mu(A)=\sharp\{x:x\in A\}.$

Next, we give some basic results.

\begin{lm}\label{09}
{\rm
Assume that (A1) and (A2) hold. Let $u\in \mathcal{E}_{p,q}$, then there exist two positive constants $c_1$ and $c_2$ such that
\begin{eqnarray*}
c_1\|u\|_{\mathcal{E}_{p,q}}\leq\|u\|_{\mathcal{E}_{p(a),q(b)}}\leq c_2\|u\|_{\mathcal{E}_{p,q}}.
\end{eqnarray*}
}
\end{lm}
\begin{proof}

By (A1) and (A2), there exists some $R$ large enough such that, for at least one point $x\in B_{R}$, $a(x)$ and $b(x)$ are positive, and, for any $x\in B_{R}^{c}$,
$\frac{\bar{a}}{2}<a(x)<\frac{3\bar{a}}{2}\,\text{and}\,\frac{\bar{b}}{2}<b(x)<\frac{3\bar{b}}{2}.$

Since $B_R$ is a finite set in $\mathbb{V}$, we have that
\begin{equation}\label{83}
d_1\|u\|_{\mathcal{E}_{p,q}(B_R)}\leq\|u\|_{\mathcal{E}_{p(a),q(b)}(B_R)}\leq d_2\|u\|_{\mathcal{E}_{p,q}(B_R)},
\end{equation}
where $d_1$ and $d_2$ are positive constants.

For any $x\in B_{R}^{c}$, clearly, there exist two positive constants $\tilde{d}_1$ and $\tilde{d}_2$ such that
\begin{equation}\label{82}
\tilde{d}_1\|u\|_{\mathcal{E}_{p,q}(B_{R}^{c})}\leq\|u\|_{\mathcal{E}_{p(a),q(b)}(B_{R}^{c})}\leq \tilde{d}_2\|u\|_{\mathcal{E}_{p,q}(B_{R}^{c})}.
\end{equation}

In addition, one gets easily that
\begin{equation}\label{81}
\|u\|_{\mathcal{E}_{p(a),q(b)}}\leq \|u\|_{\mathcal{E}_{p(a),q(b)}(B_{R})}+\|u\|_{\mathcal{E}_{p(a),q(b)}(B_{R}^{c})}\leq 2\|u\|_{\mathcal{E}_{p(a),q(b)}}.
\end{equation}

Therefore, by (\ref{83}), (\ref{82}) and (\ref{81}), one concludes that
\begin{eqnarray*}
c_1\|u\|_{\mathcal{E}_{p,q}}\leq\|u\|_{\mathcal{E}_{p(a),q(b)}}\leq c_2\|u\|_{\mathcal{E}_{p,q}},
\end{eqnarray*}
where $c_1=\frac{1}{2}\min\{d_1,\tilde{d}_1\}$ and  $c_2=2\max\{d_2,\tilde{d}_2\}$.
\end{proof}

\begin{rem}\label{lm5}
{\rm
Lemma \ref{09} implies that the spaces $\mathcal{E}_{p,q}$ and $\mathcal{E}_{p(a),q(b)}$ are actually the same one. For simplicity, in the rest of this paper, we always write the space $\mathcal{E}_{p,q}$.
}
\end{rem}

\begin{lm}\label{lm1}
{\rm(Sobolev embedding)
For any $1\leq p\leq q<+\infty$, we have that $$\mathcal{E}_{p}\subset l^{q}(\mathbb{V}).$$

}
\end{lm}
\begin{proof}
Clearly,  $\mathcal{E}_{p}\subset l^{p}(\mathbb{V}).$  In addition, we have that $l^{p}(\mathbb{V})\subset l^{q}(\mathbb{V})$ with $p< q$.
Therefore, for any $1\leq p\leq q<+\infty$, we get that $\mathcal{E}_{p}\subset l^{q}(\mathbb{V}).$

\end{proof}

We recall a well-known Br\'{e}zis-Lieb lemma \cite{BL} considered in a measure space $(\Omega, \Sigma, \mu)$, where the measure space $(\Omega, \Sigma, \mu)$ is made up of a set $\Omega$ equipped with a $\sigma-$algebra $\Sigma$ and a Borel measure $\mu:\Sigma\rightarrow[0,+\infty]$.
\begin{lm}\label{i}
{\rm(Br\'{e}zis-Lieb lemma)
Let $(\Omega, \Sigma, \mu)$ be a measure space and $\{u_n\}\subset L^{p}(\Omega, \Sigma, \mu)$ with $0<p<+\infty$. If $\{u_n\}$ is uniformly bounded in $L^{p}(\Omega)$
and $u_n\rightarrow u, \mu-$almost everywhere in $\Omega$, then we have that
\begin{eqnarray*}
\underset{n\rightarrow+\infty}{\lim}(\|u_n\|^{p}_{L^{p}(\Omega)}-\|u_n-u\|^{p}_{L^{p}(\Omega)})=\|u\|^{p}_{L^{p}(\Omega)}.
\end{eqnarray*}

}
\end{lm}

At the end of this section, we give a discrete version of the concentration compactness principle corresponding to Lions \cite{L1} on $\mathbb{R}^{N}$

\begin{lm}\label{l-0}
{\rm(Lions lemma)
Let $1\leq p<+\infty$. Assume that $\{u_n\}$ is bounded in $l^{p}(\mathbb{V})$ and
\begin{equation}\label{1-4}
\|u_{n}\|_{\infty}\rightarrow0,\qquad n\rightarrow+\infty.
\end{equation}
Then for any $p<q<+\infty$, as $n\rightarrow+\infty,$ $$u_n\rightarrow0,\qquad \text{in~} l^{q}(\mathbb{V}).$$

}
\end{lm}
\begin{proof}
For $p<q<+\infty$, by interpolation inequality, we get that
\begin{equation}\label{1-6}
\|u_n\|^{q}_{q}\leq\|u_n\|_{p}^{p}\|u_n\|_{\infty}^{q-p}.
\end{equation}
Then the result follows from (\ref{1-4}) and (\ref{1-6}).
\end{proof}

\section{Behavior of Palais-Smale sequences }
In this section, we give some crucial lemmas for the functionals $\Phi$ and $\bar{\Phi}$, which play key roles in the proofs of our second theorem.

\begin{lm}\label{l-1}
{\rm
Assume that (A1) and (A2) hold. Let $\{u_n\}\subset\mathcal{E}_{p,q}$ be a Palais-Smale sequence of the functional $\Phi$. Passing to a subsequence if necessary, there exists some $u\in\mathcal{E}_{p,q}$ such that
\begin{itemize}
\item[(i)] $u_n\rightharpoonup u,\qquad \text{in}~\mathcal{E}_{p,q}$;\\
\item[(ii)] $u_n\rightarrow u,\qquad \text{pointwise~in}~\mathbb{V}$;\\
\item[(iii)] $\Phi'(u)=0,\qquad \text{in}~\mathcal{E}^{*}_{p,q}$.
\end{itemize}
}
\end{lm}
\begin{proof}
(i)
Suppose that $\{u_n\}$ is a Palais-Smale sequence of the functional $\Phi$ at level $c$, that is, as $n\rightarrow+\infty$,
\begin{equation}\label{1-9}
\Phi(u_n)\rightarrow c\qquad \text{and} \qquad\Phi'(u_n)\rightarrow 0, \qquad \text{in~}\mathcal{E}^{*}_{p,q}.
\end{equation}
For $2\leq p<q<+\infty$, let $\theta$ be some constant satisfying $\frac{1}{q}<\theta<\frac{1}{p}$. Then it follows from (\ref{1-9}) that
\begin{eqnarray*}
(\frac{1}{2}-\theta)\|\nabla u_n\|^{2}_2+(\frac{1}{p}-\theta)\|a^{\frac{1}{p}}u_n\|^{p}_p&\leq&(\frac{1}{2}-\theta)\|\nabla u_n\|^{2}_2+(\frac{1}{p}-\theta)\|a^{\frac{1}{p}}u_n\|^{p}_p+(\theta-\frac{1}{q})\|b^{\frac{1}{q}}u_n\|^{q}_q
\\&=&\Phi(u_n)-\theta\langle\Phi'(u_n),u_n\rangle\\&\leq& c+1+\|\Phi'(u_n)\|_{\mathcal{E}^{*}_{p(a),q(b)}}\|u_n\|_{\mathcal{E}_{p(a),q(b)}}\\&\leq& c+1+C\|\Phi'(u_n)\|_{\mathcal{E}^{*}_{p(a),q(b)}}(\|\nabla u_n\|_2+\|a^{\frac{1}{p}}u_n\|_p),
\end{eqnarray*}
where we have used the fact $\|u\|_{q}\leq\|u\|_{p}$ with $p<q$ in the forth inequality. Consequently, we have that $\|u_n\|_{\mathcal{E}_{p(a)}}\leq C$. By the Sobolev embedding, we get that
$\|u_n\|_{\mathcal{E}_{p(a),q(b)}}\leq C.$ Therefore, passing to a subsequence if necessary, there exists some $u\in\mathcal{E}_{p,q}$ such that
\begin{eqnarray*}
u_n\rightharpoonup u,\qquad \text{in}~\mathcal{E}_{p,q}.
\end{eqnarray*}

(ii)
From the proof of (i), we see that $\{u_n\}$ is bounded in $l^{p}(\mathbb{V})$, and hence bounded in $l^{\infty}(\mathbb{V})$. Therefore, by diagonal principle, there exists a subsequence of $\{u_n\}$  pointwise converging to $u$.

(iii) It is sufficient to show that for any $\phi\in C_{c}(\mathbb{V})$, $\langle\Phi'(u),\phi\rangle=0.$

For any $\phi\in C_{c}(\mathbb{V})$, assume that $\text{supp}(\phi)\subseteq B_{R}$, where $R$ is a positive constant. Since $B_{R+1}$ is a finite set in $\mathbb{V}$ and $u_n\rightarrow u $ pointwise in $\mathbb{V}$ as $n\rightarrow+\infty$, we get that
\begin{eqnarray*}
\langle\Phi'(u_n),\phi\rangle-\langle\Phi'(u),\phi\rangle&=&
\frac{1}{2}\sum\limits_{x\in B_{R+1}}\sum\limits_{y\sim x}[(u_n-u)(y)-(u_n-u)(x)](\phi(y)-\phi(x))\\&&+\sum\limits_{x\in B_{R}}a(x)(|u_n(x)|^{p-2}u_{n}(x)-|u(x)|^{p-2}u(x))\phi(x)\\&&-\sum\limits_{x\in B_{R}}b(x)(|u_n(x)|^{q-2}u_{n}(x)-|u(x)|^{q-2}u(x))\phi(x)
\\&\rightarrow& 0, \qquad  n\rightarrow+\infty.
\end{eqnarray*}
Therefore, one concludes that
\begin{eqnarray*}
|\langle\Phi'(u),\phi\rangle|=\underset {n\rightarrow+\infty}{\lim}|\langle\Phi'(u_n),\phi\rangle|\leq\underset {n\rightarrow+\infty}{\lim}
\|\Phi'(u_n)\|_{\mathcal{E}^{*}_{p(a),q(b)}}\|\phi\|_{\mathcal{E}_{p(a),q(b)}}=0.
\end{eqnarray*}

\end{proof}

\begin{crl}\label{l-2}
{\rm
Let $\{u_n\}$ be a Palais-Smale sequence of the functional $\bar{\Phi}$. Passing to a subsequence if necessary, there exists some $u\in\mathcal{E}_{p,q}$ such that
\begin{itemize}
\item[(i)] $u_n\rightharpoonup u,\qquad \text{in}~\mathcal{E}_{p,q}$;\\
\item[(ii)] $u_n\rightarrow u,\qquad \text{pointwise~in}~\mathbb{V}$;\\
\item[(iii)] $\bar{\Phi}'(u)=0,\qquad \text{in}~\mathcal{E}^{*}_{p,q}$.
\end{itemize}
}
\end{crl}

\begin{lm}\label{l-3}
{\rm
Assume that (A1) and (A2) hold. Let $\{u_n\}\subset\mathcal{E}_{p,q}$ be a Palais-Smale sequence of the functional $\Phi$. Passing to a subsequence if necessary, there exists some $u\in\mathcal{E}_{p,q}$ such that
\begin{itemize}
\item[(i)] $\underset{n\rightarrow+\infty}{\lim}(\Phi(u_n)-\Phi(u_n-u))=\Phi(u)$;\\
\item[(ii)] $\underset{n\rightarrow+\infty}{\lim}\Phi'(u_n-u)=0,\qquad \text{in}~\mathcal{E}^{*}_{p,q}$.
\end{itemize}
}
\end{lm}

\begin{proof}
It follows from Lemma \ref{l-1} that
\begin{eqnarray*}
\|u_n\|_{\mathcal{E}_{p(a),q(b)}}\leq C\qquad \text{and}\qquad u_n\rightarrow u,\qquad \text{pointwise~in}~\mathbb{V}.
\end{eqnarray*}

(i)
By the Br\'{e}zis-Lieb lemma, we have that
\begin{equation}\label{88}
\|a^{\frac{1}{p}}u_n\|^{p}_p-\|a^{\frac{1}{p}}(u_n-u)\|^{p}_p=\|a^{\frac{1}{p}}u\|^{p}_p+o(1),
\end{equation}
\begin{equation}\label{87}
\|b^{\frac{1}{q}}u_n\|^{q}_q-\|b^{\frac{1}{q}}(u_n-u)\|^{q}_q=\|b^{\frac{1}{q}}u\|^{q}_q+o(1).
\end{equation}
We claim that
\begin{equation}\label{86}
\|\nabla u_n\|^{2}_2-\|\nabla (u_n-u)\|^{2}_2=\|\nabla u\|^{2}_2+o(1).
\end{equation}
In fact, let $\overline{\mathbb{E}}=\{e=(e_-,e_+):e_\pm\in \mathbb{V},\,e_-\sim e_+\}$ be a directed edge set. Then we define a function $\bar{u}:\overline{\mathbb{E}}\rightarrow\mathbb{R},\,\bar{u}(e)=u(e_+)-u(e_-)$. One can get easily that
\begin{eqnarray*}
\|\nabla u_n\|^{2}_2=\frac{1}{2}\|\bar{u}_n\|^{2}_{l^2(\overline{\mathbb{E}})}<+\infty\qquad\text{and}\qquad\bar{u}_n\rightarrow\bar{u}
\qquad\text{pointwise~in~}\overline{\mathbb{E}}.
\end{eqnarray*}
Then by the Br\'{e}zis-Lieb lemma, we obtain that $\underset{n\rightarrow+\infty}{\lim}(\|\bar{u}_n\|^{2}_{L^{2}(\overline{\mathbb{E}})}-
\|\bar{u}_n-\bar{u}\|^{2}_{L^{2}(\overline{\mathbb{E}})})=\|\bar{u}\|^{2}_{L^{2}(\overline{\mathbb{E}})},$ which means that (\ref{86}). Hence, we complete the claim.

Therefore, by (\ref{88}), (\ref{87}) and (\ref{86}), we obtain that $\underset{n\rightarrow+\infty}{\lim}(\Phi(u_n)-\Phi(u_n-u))=\Phi(u).$

(ii) For any $\phi\in C_{c}(\mathbb{V})$, assume that $\text{supp}(\phi)\subseteq B_{R}$, where $R$ is a positive constant. Since $B_{R+1}$ is a finite set in $\mathbb{V}$ and $u_n\rightarrow u $ pointwise in $\mathbb{V}$ as $n\rightarrow+\infty$, we get that
\begin{eqnarray*}
|\langle\Phi'(u_n-u),\phi\rangle|&\leq&
\sum\limits_{x\in B_{R+1}}|\nabla(u_n-u)||\nabla\phi|+\sum\limits_{x\in B_{R}}a(x)|u_n-u|^{p-1}|\phi|\\&&+\sum\limits_{x\in B_{R}}b(x)|u_n-u|^{q-1}|\phi|\\&\leq&\|\nabla(u_n-u)\|_{l^{2}(B_{R+1})}\|\nabla\phi\|_{2}+\|a^{\frac{1}{p}}(u_n-u)\|^{p-1}_{l^{p}(B_R)}
\|a^{\frac{1}{p}}\phi\|_{p}\\&&+\|b^{\frac{1}{q}}(u_n-u)\|^{q-1}_{l^{q}(B_R)}
\|b^{\frac{1}{q}}\phi\|_{q}\\&\leq&C\xi_{n}\|\phi\|_{\mathcal{E}_{p(a),q(b)}},
\end{eqnarray*}
where $C$ is a constant not depending on $n$ and $\xi_{n}\rightarrow 0$ as $n\rightarrow+\infty$.
Therefore, we get that
$$\underset{n\rightarrow+\infty}{\lim}\|\Phi'(u_n-u)\|_{\mathcal{E}^{*}_{p(a),q(b)}}=\underset{n\rightarrow+\infty}{\lim}
\underset{\|\phi\|_{\mathcal{E}_{p(a),q(b)}}=1}{\sup}|\langle\Phi'(u_n-u),\phi\rangle|=0.$$

\end{proof}


\begin{crl}\label{c0}
{\rm
Let $\{u_n\}\subset\mathcal{E}_{p,q}$ be a Palais-Smale sequence of the functional $\bar{\Phi}$. Passing to a subsequence if necessary, there exists some $u\in\mathcal{E}_{p,q}$ such that
\begin{itemize}
\item[(i)] $\underset{n\rightarrow+\infty}{\lim}(\bar{\Phi}(u_n)-\bar{\Phi}(u_n-u))=\bar{\Phi}(u)$;\\
\item[(ii)] $\underset{n\rightarrow+\infty}{\lim}\bar{\Phi}'(u_n-u)=0,\qquad \text{in}~\mathcal{E}^{*}_{p,q}$.
\end{itemize}
}
\end{crl}

Finally, we give a relation between the functionals $\Phi$ and $\bar{\Phi}$.

\begin{lm}\label{l-5}
{\rm
Assume that (A1) and (A2) hold. Let a sequence $\{u_n\}\subset\mathcal{E}_{p,q}$ satisfy

\begin{equation}\label{2-6}
u_n\rightharpoonup0,\qquad \text{in}~\mathcal{E}_{p,q}.
\end{equation}
Then we have that

\begin{itemize}
\item[(i)] $\underset{n\rightarrow+\infty}{\lim}|\Phi(u_n)-\bar{\Phi}(u_n)|=0$;\\
\item[(ii)] $\underset{n\rightarrow+\infty}{\lim}(\Phi'(u_n)-\bar{\Phi}'(u_n))=0,\qquad \text{in~}\mathcal{E}^{*}_{p,q}$.
\end{itemize}
}
\end{lm}

\begin{proof}
By (\ref{2-6}), one sees that $\{u_n\}$ is bounded in $\mathcal{E}_{p,q}$ and $u_n\rightarrow 0$ pointwise in $\mathbb{V}$ as $n\rightarrow+\infty$.
For any $R>0$, since $B_R$ is a finite set, we have that, for any $s\geq 1$,
\begin{equation}\label{84}
u_n\rightarrow 0,\qquad \text{in}~l^{s}(B_{R}).
\end{equation}

By (A1) and (A2), for any $\varepsilon>0$, there exists $R_{\varepsilon}>0$ such that
\begin{eqnarray*}
|a(x)-\bar{a}|<\varepsilon \qquad \text{and}\qquad |b(x)-\bar{b}|<\varepsilon, \qquad x\in B_{R_{\varepsilon}^{c}}.
\end{eqnarray*}

(i) Given $\varepsilon$ and $R_\varepsilon$, then, by (\ref{84}), there exists $n_\varepsilon$ such that, for all $n\geq n_\varepsilon$,
\begin{eqnarray*}
|\int_{\mathbb{V}}(a(x)-\bar{a})|u_n|^p \,d\mu|&\leq&\int_{B_{R_{\varepsilon}^{c}}}|a(x)-\bar{a}||u_n|^p \,d\mu+\int_{B_{R_{\varepsilon}}}|a(x)-\bar{a}||u_n|^p \,d\mu\\&\leq&
\varepsilon\|u_n\|_{p}^{p}+C\|u_n\|^{p}_{l^{p}(B_{R_\varepsilon})}\\&\leq&C\varepsilon
\end{eqnarray*}
and
\begin{eqnarray*}
|\int_{\mathbb{V}}(b(x)-\bar{b})|u_n|^q d\mu|&\leq&\int_{B_{R_{\varepsilon}^{c}}}|b(x)-\bar{b}||u_n|^q d\mu+\int_{B_{R_{\varepsilon}}}|b(x)-\bar{b}||u_n|^q d\mu\\&\leq&
\varepsilon\|u_n\|_{q}^{q}+C\|u_n\|^{q}_{l^{q}(B_{R_{\varepsilon}})}\\&\leq&C\varepsilon,
\end{eqnarray*}
where $C$ does not depend on $n$ and $\varepsilon$. Therefore, by letting $\varepsilon\rightarrow0$, we get that
$$\underset{n\rightarrow+\infty}{\lim}|\Phi(u_n)-\bar{\Phi}(u_n)|=0.$$

(ii) By similar arguments as in the proof of (i), one can obtain that
$$\underset{n\rightarrow+\infty}{\lim}(\Phi'(u_n)-\bar{\Phi}'(u_n))=0,\qquad \text{in~}\mathcal{E}^{*}_{p,q}.$$

\end{proof}

\
\section{Proofs of the main results}
In this section, we are devoted to prove Theorem \ref{t-9} and \ref{t-0} by using the Br\'{e}zis-Lieb lemma and concentration compactness principle.
\
\

First, we prove the existence of positive solutions to the equation $$-\Delta u+\tilde{a}|u|^{p-2}u-\tilde{b}|u|^{q-2}u=0.$$

\
\

\noindent {\bf Proof of Theorem  \ref{t-9}:}
For any $\tilde{a}>0$, we define functionals $J_1, J_2 : \mathcal{E}_{p,q}\rightarrow\R$ in the form
\begin{eqnarray*}
J_{1}(u)=\frac{1}{2}\|\nabla u\|_{2}^{2}+\frac{\tilde{a}}{p}\|u\|_{p}^{p},\qquad
J_{2}(u)=\beta\|u\|_{q}^{q},
\end{eqnarray*}
where $\beta$ is a positive parameter.

\noindent {\bf Step I:} We show that the variational problem
\begin{equation}\label{7-1}
\lambda_{0}:=\inf\{J_1(u):u\in \mathcal{E}_{p,q},J_2(u)=1\}>0
\end{equation}
has a nonnegative solution $u_{(0)}\in\mathcal{E}_{p,q}$.

Without loss of generality, we may assume that $u\geq0$ since $J_1(|u|)\leq J_1(u)$ and $J_2(|u|)=J_2(u)$. Let $\{u_n\}\subset \mathcal{E}_{p,q}$ be a nonnegative minimizing sequence satisfying
\begin{eqnarray*}\label{7-2}
J_{2}(u_{n})=1,\qquad \underset{n\rightarrow+\infty}{\lim}J_{1}(u_{n})=\lambda_{0},
\end{eqnarray*}
which implies that $\|u_n\|_{\mathcal{E}_{p,q}}\leq C$. Hence, passing to a subsequence if necessary, we have that
\begin{eqnarray*}\label{7-8}
\begin{array}{ll}
u_n\rightharpoonup u_{(0)},\qquad \text{in}~\mathcal{E}_{p,q},\\
u_n\rightarrow u_{(0)}, \qquad \text{pointwise~in}~\mathbb{V}.
\end{array}
\end{eqnarray*}
 By the Fatou lemma, we obtain that
\begin{eqnarray*}
\begin{array}{ll}
J_{1}(u_{(0)})\leq\underset{n\rightarrow+\infty}{\lim}J_{1}(u_{n})=\lambda_{0},\\
J_{2}(u_{(0)})\leq\underset{n\rightarrow+\infty}{\lim}J_{2}(u_{n})=1.
\end{array}
\end{eqnarray*}

We claim that $J_{2}(u_{(0)})=1.$ This implies that
$$\underset{n\rightarrow+\infty}{\lim}J_{1}(u_{n})=J_{1}(u_{(0)})=\lambda_{0}.$$ Then the proof of Step I is completed.

We prove the claim by contradiction. Suppose that there exists some constant $\theta_{0}$ such that $\theta_{0}=J_{2}(u_{(0)})<1.$
Since $p<q<+\infty$, by interpolation inequality, for $n$ large enough, we get that
\begin{eqnarray*}
1=J_2(u_n)=\beta\|u_n\|_{q}^{q}\leq\beta\|u_n\|_{p}^{p}\|u_n\|_{\infty}^{q-p}\leq\beta(1+\lambda_{0})\|u_n\|_{\infty}^{q-p}.
\end{eqnarray*}
Consequently,
\begin{eqnarray*}
\|u_n\|_{\infty}\geq[\frac{1}{\beta(1+\lambda_{0})}]^{\frac{1}{q-p}}\triangleq\sigma_{0}>0.
\end{eqnarray*}
After a subsequence and translation if necessary, we may assume that
\begin{eqnarray*}
u_n(0)=\|u_n\|_{\infty}\geq\sigma_{0}>0.
\end{eqnarray*}
Since $u_n(x)\rightarrow u_{(0)}(x)$ pointwise in $\mathbb{V}$, we have that $u_{(0)}(0)\geq\sigma_{0}>0.$ This yields that
 \begin{eqnarray*}
\theta_{0}=J_{2}(u_{(0)})\geq\beta\sigma_{0}^q=[\frac{1}{\beta^{\frac{p}{q}}(1+\lambda_{0})}]^{\frac{q}{q-p}}>0.
\end{eqnarray*}

Denote
\begin{eqnarray*}\label{1}
v^{1}_{n}(x)=u_{n}(x)-u_{(0)}(x).
\end{eqnarray*}
Then we have that
\begin{equation}\label{7-7}
\begin{array}{ll}
v^{1}_n\rightharpoonup 0, \qquad \text{in}~\mathcal{E}_{p,q},\\
v^{1}_n\rightarrow 0, \qquad \text{pointwise~in}~\mathbb{V}.
\end{array}
\end{equation}
According to the Br\'{e}zis-Lieb lemma, we obtain that
\begin{equation}\label{7-3}
\begin{array}{ll}
J_{1}(u_{n})=J_{1}(u_{(0)})+J_{1}(v^{1}_{n})+o(1),\\
J_{2}(u_{n})=J_{2}(u_{(0)})+J_{2}(v^{1}_n)+o(1).
\end{array}
\end{equation}
The above results imply, for $n$ large enough, that
\begin{eqnarray*}\label{7-5}
\begin{array}{ll}
J_{1}(v^{1}_{n})\leq(1+\lambda_{0}),\\
J_{2}(v^{1}_n)=1-\theta_{0}+o(1)\geq\frac{1}{2}(1-\theta_{0}).
\end{array}
\end{eqnarray*}
Again, by interpolation inequality, for $n$ large enough, one has that
\begin{eqnarray*}
\frac{1}{2}(1-\theta_{0})\leq J_{2}(v^{1}_n)=\beta\|v^{1}_n\|^{q}_{q}\leq\beta\|v^{1}_n\|^{p}_{p}\|v^{1}_n\|^{q-p}_{\infty}\leq\beta(1+\lambda_{0})\|v^{1}_n\|^{q-p}_{\infty},
\end{eqnarray*}
and hence
\begin{eqnarray*}
\|v^{1}_n\|_{\infty}\geq[\frac{(1-\theta_{0})}{2\beta(1+\lambda_{0})}]^{\frac{1}{q-p}}\triangleq\sigma_{1}>0.
\end{eqnarray*}
Then there exists a sequence $\{x^{1}_{n}\}\subset\mathbb{V}$ such that $|v^{1}_n(x^1_n)|=\|v^{1}_n\|_{\infty}\geq\sigma_{1}>0.$

Let
\begin{eqnarray*}\label{2}
u_{n(1)}(x)=v^{1}_n(x+x^{1}_n).
\end{eqnarray*}
We assume that for some $u_{(1)}\in\mathcal{E}_{p,q}$, as $n\rightarrow+\infty$,
\begin{equation}\label{7-10}
\begin{array}{ll}
u_{n(1)}\rightharpoonup u_{(1)}, \qquad \text{in}~\mathcal{E}_{p,q},\\
u_{n(1)}\rightarrow u_{(1)}, \qquad \text{pointwise~in}~\mathbb{V}.
\end{array}
\end{equation}
Clearly, $|u_{(1)}(0)|\geq\sigma_{1}>0$. Hence, by (\ref{7-7}) and (\ref{7-10}), one sees that $|x^{1}_n|\rightarrow+\infty$ as $n\rightarrow+\infty$.
We also see that
\begin{eqnarray*}
\theta_{1}\triangleq J_{2}(u_{(1)})\geq\beta\sigma_{1}^{q}=[\frac{(1-\theta_{0})}{2\beta^{\frac{p}{q}}(1+\lambda_{0})}]^{\frac{q}{q-p}}>0.
\end{eqnarray*}
Note that the invariance of functionals with respect to translations, we have that
\begin{equation}\label{5}
J_{1}(u_{n(1)})=J_{1}(v^{1}_n), \qquad J_{2}(u_{n(1)})=J_{2}(v^{1}_n).
\end{equation}

Denote
\begin{equation}\label{3}
v^{2}_{n}(x)=u_{n}(x)-u_{(0)}(x)-u_{(1)}(x-x^{1}_n).
\end{equation}
Then it follows from (\ref{7-10}) that
\begin{eqnarray*}\label{7-17}
\begin{array}{ll}
v^{2}_n\rightharpoonup 0, \qquad \text{in}~\mathcal{E}_{p,q},\\
v^{2}_n\rightarrow 0, \qquad \text{pointwise~in}~\mathbb{V}.
\end{array}
\end{eqnarray*}
By (\ref{7-3}), (\ref{5}) and the Br\'{e}zis-Lieb lemma, we obtain that
\begin{equation}\label{7-12}
\begin{array}{ll}
J_{1}(u_{n})=J_{1}(u_{(0)})+J_{1}(u_{(1)})+J_{1}(v^{2}_n)+o(1),\\
1=J_{2}(u_{(0)})+J_{2}(u_{(1)})+J_{2}(v^{2}_n)+o(1)=\theta_0+\theta_{1}+J_{2}(v^{2}_n)+o(1).
\end{array}
\end{equation}

 If $\theta_0+\theta_{1}=1$, i.e. $J_{2}(u_{(0)})+J_2(u_{(1)})=1$, one gets that
\begin{eqnarray*}
J_{1}(u_{(0)})+J_{1}(u_{(1)})\leq\lambda_{0},\qquad J_{2}(v^{2}_n)=o(1).
\end{eqnarray*}
As a consequence, at least one of $\frac{J_{1}(u_{(0)})}{J_{2}(u_{(0)})}$ and $\frac{J_{1}(u_{(1)})}{J_{2}(u_{(1)})}$ is less than $\lambda_{0}$.
Without loss of generality, we assume that $\frac{J_{1}(u_{(1)})}{J_{2}(u_{(1)})}\leq\lambda_{0}$. Let $\bar{u}_{(1)}=\theta_{1}^{-\frac{1}{q}}u_{(1)}$, then
$J_{2}(\bar{u}_{(1)})=1$ and
\begin{eqnarray*}
\frac{J_{1}(\bar{u}_{(1)})}{J_{2}(\bar{u}_{(1)})}&\leq&\frac{\theta_{1}^{\frac{-2}{q}}\frac{1}{2}\|\nabla u_{(1)}\|^{2}_{2}+\theta_{1}^{\frac{-p}{q}}\frac{\bar{a}}{p}\|u_{(1)}\|_{p}^{p}}{\theta_{1}^{-1}J_{2}(u_{(1)})}\\&=&\frac{\theta_{1}^{\frac{q-2}{q}}\frac{1}{2}\|\nabla u_{(1)}\|^{2}_{2}+\theta_{1}^{\frac{q-p}{q}}\frac{\bar{a}}{p}\|u_{(1)}\|_{p}^{p}}{J_{2}(u_{(1)})}\\&<&\frac{J_{1}(u_{(1)})}{J_{2}(u_{(1)})}\leq\lambda_{0}.
\end{eqnarray*}
This yields a contradiction.

 If $\theta_0+\theta_{1}<1$, from (\ref{7-12}), for $n$ large enough, one has that
\begin{eqnarray*}
\frac{1}{2}(1-\theta_{0}-\theta_{1})\leq J_{2}(v^{2}_n)=\beta\|v^{2}_n\|^{q}_{q}\leq\beta\|v^{2}_n\|^{p}_{p}\|v^{2}_n\|^{q-p}_{\infty}\leq\beta(1+\lambda_{0})\|v^{2}_n\|^{q-p}_{\infty},
\end{eqnarray*}
and hence
\begin{eqnarray*}
\|v^{2}_n\|_{\infty}\geq[\frac{(1-\theta_{0}-\theta_{1})}{2\beta(1+\lambda_{0})}]^{\frac{1}{q-p}}\triangleq\sigma_{2}>0.
\end{eqnarray*}
Then there exists a sequence $\{x^{2}_{n}\}\subset\mathbb{V}$ such that $|v^{2}_n(x^{2}_{n})|=\|v^{2}_n\|_{\infty}\geq\sigma_{2}>0.$

Let
\begin{equation}\label{4}
u_{n(2)}(x)=v^{2}_n(x+x^{2}_{n}).
\end{equation}
There exists some $u_{(2)}\in\mathcal{E}_{p,q}$, as $n\rightarrow+\infty$,  such that
\begin{eqnarray*}\label{7-15}
\begin{array}{ll}
u_{n(2)}\rightharpoonup u_{(2)}, \qquad \text{in}~\mathcal{E}_{p,q},\\
u_{n(2)}\rightarrow u_{(2)}, \qquad \text{pointwise~in}~\mathbb{V},
\end{array}
\end{eqnarray*}
where $|u_{(2)}(0)|\geq\sigma_{2}>0$. Similarly, we have that $|x^{2}_{n}|\rightarrow+\infty$ as $n\rightarrow+\infty$.
In fact, we also have that $|x^{2}_{n}-x^{1}_n|\rightarrow+\infty$ as $n\rightarrow+\infty$. Suppose it is not true, then $\{x^{2}_n-x^{1}_n\}$ is bounded in $\mathbb{V}$. Thus there exists a point $z_{0}\in\mathbb{V}$ such that $(x^{2}_n-x^{1}_n)\rightarrow z_0$ as $n\rightarrow+\infty.$

For any $x\in \mathbb{V}$, it follows from (\ref{3}) and (\ref{4}) that
\begin{eqnarray*}
u_n(x)=u_{(0)}(x)+u_{(1)}(x-x^{1}_n)+u_{n(2)}(x-x^{2}_n).
\end{eqnarray*}
In particular, we have that
\begin{eqnarray*}
u_{n(2)}(x-x^{2}_n+x^{1}_n)=u_n(x+x^{1}_n)-u_{(0)}(x+x^{1}_n)-u_{(1)}(x).
\end{eqnarray*}
By taking $x=z_0$, we get that
\begin{equation}\label{7-18}
u_{n(2)}(z_0-x^{2}_n+x^{1}_n)=u_{n(1)}(z_0)-u_{(1)}(z_0).
\end{equation}
Since $(x^{2}_n-x^{1}_n)\rightarrow z_0$ as $n\rightarrow+\infty$ and $|u_{(2)}(0)|>0$, one sees that $u_{n(2)}(z_0-x^{2}_n+x^{1}_n)\nrightarrow 0$ in the left hand side of (\ref{7-18}). While the right hand side tends to zero as $n\rightarrow+\infty$ due to (\ref{7-10}). We get a contradiction.
Furthermore, we have that
\begin{eqnarray*}
\theta_{2}\triangleq J_{2}(u_{(2)})\geq\beta\sigma_{2}^{q}=[\frac{(1-\theta_{0}-\theta_{1})}{2\beta^{\frac{p}{q}}(1+\lambda_{0})}]^{\frac{q}{q-p}}>0.
\end{eqnarray*}

We repeat the process. By iteration, for $\{v^{j}_{n}\}$ with $j\geq 2$, if $\underset {i=0}{\overset{j-1}{\Sigma}}\theta_{i}=1$, we get a contradiction; if $\underset {i=0}{\overset{j-1}{\Sigma}}\theta_{i}<1$, we get that
$$v^{j+1}_{n}(x)=u_{n}(x)-u_{(0)}(x)-\underset {i=1}{\overset{j}{\Sigma}}u_{(i)}(x-x^{i}_n)$$
with $|x^{i}_n|\rightarrow+\infty$ and $|x^{j}_n-x^{i}_n|\rightarrow+\infty,\,j\neq i$ as $n\rightarrow+\infty$, and
\begin{eqnarray*}
\theta_{j}\triangleq J_{2}(u_{(j)}) \geq\beta\sigma^{q}_{j}=[\frac{1-\underset {i=0}{\overset{j-1}{\Sigma}}\theta_{i}}{2\beta^{\frac{p}{q}}(1+\lambda_{0})}]^{\frac{q}{q-p}}>0.
\end{eqnarray*}
If the procedure terminate at $\theta_0+\theta_{1}+\cdots+\theta_{k}=1$ with some integer $k\geq 1$, we get a contradiction. If not, then there exists a sequence $\{\theta_{i}\},\, i=0,1,\cdots$, such that $\underset {i=0}{\overset{\infty}{\Sigma}}\theta_{i}\leq 1$. We prove that $\underset {i=0}{\overset{\infty}{\Sigma}}\theta_{i}= 1$. By contradiction, assume that $\underset {i=0}{\overset{\infty}{\Sigma}}\theta_{i}=\delta< 1$, without loss of generality, let $\underset {i=0}{\overset{n}{\Sigma}}\theta_{i}=\delta-\varepsilon_{0}<\delta$, where $\varepsilon_{0}>0$ is an arbitrary constant, then $\theta_{n+1}\geq[\frac{1-(\delta-\varepsilon_{0})}{2\beta^{\frac{p}{q}}(1+\lambda_{0})}]^{\frac{q}{q-p}}>0.$ By taking $\varepsilon_{0}$ small enough, we get that $\underset {i=0}{\overset{n+1}{\Sigma}}\theta_{i}>\delta$, which contradicts $\underset {i=0}{\overset{\infty}{\Sigma}}\theta_{i}=\delta$. Hence $\underset {i=0}{\overset{\infty}{\Sigma}}\theta_{i}= 1$, which yields a decomposition of $\{u_n\}$ with infinitely many bubbles
$$u_n(x)=u_{(0)}(x)+\underset {i=1}{\overset{\infty}{\Sigma}}u_{(i)}(x-x^{i}_n),$$
where $|x^{i}_n|\rightarrow+\infty$ and $|x^{j}_n-x^{i}_n|\rightarrow+\infty,\,j\neq i$ as $n\rightarrow+\infty$. Similar to the case $\theta_0+\theta_{1}=1$, we also get a contradiction. Hence $J_{2}(u_{(0)})=\theta_{0}<1$ is impossible. We complete the claim $J_{2}(u_{(0)})=1$.

\noindent {\bf Step II:}  We prove that for any $\tilde{a}>0$, there are infinitely many positive constants $\tilde{b}$ such that the equation (\ref{0-0}) has a positive solution.

By the Lagrange theorem, there exists $\lambda>0$ such that
\begin{eqnarray*}
\int_{\mathbb{V}}|\nabla u_{(0)}|^{2}\,d\mu+\tilde{a}\int_{\mathbb{V}}u_{(0)}^{p}\,d\mu=\lambda q\beta\int_{\mathbb{V}}u_{(0)}^{q}\,d\mu=\lambda q.
\end{eqnarray*}

On the one hand,
\begin{eqnarray*}
\lambda&=&\frac{1}{q}\int_{\mathbb{V}}|\nabla u_{(0)}|^{2}\,d\mu+\frac{\tilde{a}}{q}\int_{\mathbb{V}}u_{(0)}^{p}\,d\mu\\&\geq&
\frac{\tilde{a}}{q}\int_{\mathbb{V}}u_{(0)}^{p}\,d\mu\\&\geq&\frac{\tilde{a}}{q}(\int_{\mathbb{V}}u_{(0)}^{q}\,d\mu)^{\frac{p}{q}}\\&=&
\frac{\tilde{a}}{q}\frac{1}{\beta^{\frac{p}{q}}},
\end{eqnarray*}
where we have used the fact that $\|u\|_q\leq\|u\|_p$ with $p<q<+\infty$.

On the other hand, let
\begin{eqnarray*}
v_0(x)=
\left\{
\begin{array}{ll}
\frac{1}{\beta^{\frac{1}{q}}}, &x=0,\\
0,  &x\neq0.
\end{array}
\right.
\end{eqnarray*}
Then $v_0\in\mathcal{E}_{p,q}$ and $J_2(v_0)=1$. Since $J_{1}(u_{(0)})=\lambda_{0}$ and $J_{2}(u_{(0)})=1$, we get that
$J_{1}(u_{(0)})\leq J_{1}(v_{0}).$ Therefore, for $2\leq p<q<+\infty$,
\begin{eqnarray*}
\lambda&=&\frac{1}{q}\int_{\mathbb{V}}|\nabla u_{(0)}|^{2}\,d\mu+\frac{\tilde{a}}{q}\int_{\mathbb{V}}u_{(0)}^{p}\,d\mu\\&\leq&
\frac{1}{2}\int_{\mathbb{V}}|\nabla u_{(0)}|^{2}\,d\mu+\frac{\tilde{a}}{p}\int_{\mathbb{V}}u_{(0)}^{p}\,d\mu\\&\leq&J_{1}(v_{0})
\\&\leq&\frac{C}{\beta^{\frac{2}{q}}}+\frac{\tilde{a}}{p}\frac{1}{\beta^{\frac{p}{q}}}.
\end{eqnarray*}
As a consequence, we have the following estimate
\begin{eqnarray*}
C\beta^{\frac{q-p}{q}}\leq\lambda\beta q\leq C(\beta^{\frac{q-2}{q}}+\beta^{\frac{q-p}{q}}).
\end{eqnarray*}
Denote $\tilde{b}=\lambda\beta q$. Clearly, we have that $\tilde{b}\rightarrow 0$ as $\beta\rightarrow 0^{+}$ and $\tilde{b}\rightarrow+\infty$ as $\beta\rightarrow+\infty$. This implies that for any $\tilde{a}>0$, there are infinitely many positive constants $\tilde{b}$ such that the equation
\begin{eqnarray*}
-\Delta u+\tilde{a}|u|^{p-2}u-\tilde{b}|u|^{q-2}u=0
\end{eqnarray*}
has a nonnegative solution $u_{(0)}\in\mathcal{E}_{p,q}$. By maximum principle, we get that $u_{(0)}(x)>0$ for any $x\in\mathbb{V}$.
\qed

\
\

Finally, we  give a decomposition of bounded Palais-Smale sequence for the functional $\Phi$ related to the equation (\ref{1-0}).

\
\

{\bf Proof of Theorem \ref{t-0}}:
Let $\{u_n\}\subset\mathcal{E}_{p,q}$ be a Palais-Smale sequence of $\Phi$ at level $c\in\R$. By Lemma \ref{l-1}, there exists some $u_{(0)}\in\mathcal{E}_{p,q}$ such that
\begin{eqnarray*}\label{3-0}
\begin{array}{ll}
u_n\rightharpoonup u_{(0)},\qquad \text{in}~\mathcal{E}_{p,q},\\
u_n\rightarrow u_{(0)}, \qquad \text{pointwise~in}~\mathbb{V},\\
\Phi'(u_{(0)})=0,\qquad \text{in}~\mathcal{E}^{*}_{p,q}.
\end{array}
\end{eqnarray*}
The third equality tells us that $u_{(0)}$ is a solution of the equation (\ref{1-0}).

Denote
\begin{eqnarray*}\label{3-9}
v^{1}_n(x)=u_n(x)-u_{(0)}(x).
\end{eqnarray*}
Then we have that
\begin{eqnarray*}
\begin{array}{ll}
v^{1}_n\rightharpoonup 0,\qquad \text{in}~\mathcal{E}_{p,q},\\
v^{1}_n\rightarrow 0, \qquad \text{pointwise~in}~\mathbb{V}.
\end{array}
\end{eqnarray*}
According to Lemma \ref{l-3}, one obtains that
\begin{equation}\label{3-21}
\begin{array}{ll}
\Phi(v^{1}_n)=c-\Phi(u_{(0)})+o(1),\\
\Phi'(v^{1}_n)=o(1),\qquad \text{in}~\mathcal{E}^{*}_{p,q}.
\end{array}
\end{equation}
By Lemma \ref{l-5}, one gets that
\begin{equation}\label{3-3}
\begin{array}{ll}
\bar{\Phi}(v^{1}_n)=c-\Phi(u_{(0)})+o(1),\\
\bar{\Phi}'(v^{1}_n)=o(1),\qquad \text{in}~\mathcal{E}^{*}_{p,q}.
\end{array}
\end{equation}
For $\{v_{n}^{1}\}$, two cases are possible:

{\bf Case 1.} $\underset{n\rightarrow+\infty}{\overline{\lim}}\|v^1_{n}\|_{\infty}=0$.  By the boundedness of $\{v^1_{n}\}$ in $\mathcal{E}_{p,q}$ and Lions lemma, we have that $\|b^{\frac{1}{q}}v^{1}_n\|_{q}\rightarrow 0,$ as $ n\rightarrow+\infty.$ Then
combined with (\ref{3-21}), we get that
\begin{eqnarray*}
\|\nabla v^{1}_n\|^{2}_{2}+\|a^{\frac{1}{p}}v^{1}_n\|^{p}_{p}&=& \langle\Phi'(v^{1}_n),v^{1}_{n}\rangle+\|b^{\frac{1}{q}}v^{1}_n\|^{q}_{q}\\&\leq&
\|\Phi'(v^{1}_n)\|_{\mathcal{E}^{*}_{p(a),q(b)}}\|v^{1}_n\|_{\mathcal{E}_{p(a),q(b)}}+\|b^{\frac{1}{q}}v^{1}_n\|^{q}_{q}\\&\leq&
C\|\Phi'(v^{1}_n)\|_{\mathcal{E}^{*}_{p(a),q(b)}}+\|b^{\frac{1}{q}}v^{1}_n\|^{q}_{q}\\&\rightarrow&0, \qquad  n\rightarrow+\infty.
\end{eqnarray*}
Consequently, as $n\rightarrow+\infty$, one concludes that $v^{1}_n\rightarrow 0$ in $\mathcal{E}_{p,q},$ and hence $c=\Phi(u_{(0)}).$

The proof ends with $k=0$.

{\bf Case 2.} $\underset{n\rightarrow+\infty}{\underline{\lim}}\|v_{n}\|_{\infty}=\delta>0$.
Then there exists a sequence $\{y^{1}_{n}\}\subset\mathbb{V}$ such that $|v^{1}_{n}(y^{1}_{n})|\geq\frac{\delta}{2}>0.$

Let
\begin{eqnarray*}
u_{n(1)}(x)=v^{1}_{n}(x+y^{1}_{n}).
\end{eqnarray*}
We may assume that for some $u_{(1)}\in\mathcal{E}_{p,q}$, as $n\rightarrow+\infty$,
\begin{eqnarray*}\label{3-10}
\begin{array}{ll}
u_{n(1)}\rightharpoonup u_{(1)}, \qquad \text{in}~\mathcal{E}_{p,q},\\
u_{n(1)}\rightarrow u_{(1)}, \qquad \text{pointwise~in}~\mathbb{V}.
\end{array}
\end{eqnarray*}
Clearly, $|u_{(1)}(0)|\geq\frac{\delta}{2}>0$. Since $v^{1}_n\rightarrow 0$ pointwise in $\mathbb{V}$, we see that $|y^{1}_{n}|\rightarrow+\infty$ as $n\rightarrow+\infty$.

By the invariance of $\bar{\Phi}$ with respect to translations and (\ref{3-3}), we have that
\begin{eqnarray*}\label{3-7}
\begin{array}{ll}
\bar{\Phi}(u_{n(1)})=c-\Phi(u_{(0)})+o(1),\\
\bar{\Phi}'(u_{n(1)})=o(1),\qquad \text{in}~\mathcal{E}^{*}_{p,q}.
\end{array}
\end{eqnarray*}
Then it follows from Corollary \ref{l-2} that
\begin{eqnarray*}
\bar{\Phi}'(u_{(1)})=0,\qquad \text{in}~\mathcal{E}^{*}_{p,q}.
\end{eqnarray*}
This means that $u_{(1)}$ is a nontrivial solution of the equation (\ref{1-2}).

Denote
\begin{eqnarray*}\label{3-11}
v^{2}_{n}(x)=u_{n}(x)-u_{(0)}(x)-u_{(1)}(x-y^{1}_n).
\end{eqnarray*}
Then we have that
\begin{eqnarray*}\label{3-17}
\begin{array}{ll}
v^{2}_n\rightharpoonup 0, \qquad \text{in}~\mathcal{E}_{p,q},\\
v^{2}_n\rightarrow 0, \qquad \text{pointwise~in}~\mathbb{V}.
\end{array}
\end{eqnarray*}
According to Corollary \ref{c0}, we obtain that
\begin{equation}\label{4-1}
\begin{array}{ll}
\bar{\Phi}(v^{2}_n)=c-\Phi(u_{(0)})-\bar{\Phi}(u_{(1)})+o(1),\\
\bar{\Phi}'(v^{2}_n)=o(1),\qquad \text{in}~\mathcal{E}^{*}_{p,q}.
\end{array}
\end{equation}
For $\{v^{2}_n\}$, two cases are possible. If the vanishing case occurs for $\|v^{2}_n\|_{\infty}$, by similar arguments as in {\bf Case 1}, one concludes that
\begin{eqnarray*}
\begin{array}{ll}
v^{2}_n\rightarrow 0,\qquad \text{in}~\mathcal{E}_{p,q},\\
c=\Phi(u_{(0)})+\bar{\Phi}(u_{(1)}).
\end{array}
\end{eqnarray*}
The proof is completed with $k=1$.

If the non-vanishing occurs for $\|v^{2}_n\|_{\infty}$, by analogous discussions as in {\bf Case 2} and Theorem \ref{t-9}, there exists a sequence $\{y^{2}_{n}\}\subset\mathbb{V}$ such that $|y_n^2|\rightarrow+\infty$ and $|y_n^2-y_n^1|\rightarrow+\infty$ as $n\rightarrow+\infty$,
and a nontrivial function $u_{(2)}\in\mathcal{E}_{p,q}$  satisfying the equation (\ref{1-2}).

We repeat the process. By iteration, for $\{v^{j}_n\}$ with $j\in\mathbb{N}^{+}$, if the vanishing case occurs for $\|v^{j}_n\|_{\infty}$, we get that $v^{j}_n\rightarrow0$ in $\mathcal{E}_{p,q}$ and $c=\Phi(u_{(0)})+\underset {i=1}{\overset{j-1}{\Sigma}}\bar{\Phi}(u_{(i)})$; if the non-vanishing case occurs for $\|v^{j}_n\|_{\infty}$, we get that
$$v^{j+1}_{n}(x)=u_{n}(x)-u_{(0)}(x)-\underset {i=1}{\overset{j}{\Sigma}}u_{(i)}(x-y^{i}_n)$$
with $|y^{i}_{n}|\rightarrow+\infty$ and $|y^{j}_{n}-y^{i}_{n}|\rightarrow+\infty,\,j\neq i$ as $n\rightarrow+\infty$, and
\begin{eqnarray*}
\begin{array}{ll}
\bar{\Phi}(v^{j+1}_n)=c-\Phi(u_{(0)})-\underset {i=1}{\overset{j}{\Sigma}}\bar{\Phi}(u_{(i)})+o(1),\\
\bar{\Phi}'(v^{j+1}_n)=o(1),\qquad \text{in}~\mathcal{E}^{*}_{p,q}.
\end{array}
\end{eqnarray*}

Since, for every $j$, $\bar{\Phi}(u_{(j)})\geq I_{0}(\bar{\Phi})>0$, the iteration must stop after finite steps, for example at some $k$. That is, the last Palais-Smale sequence denoted by $$v^{k+1}_{n}=u_{n}(x)-u_{(0)}(x)-\underset {i=1}{\overset{k}{\Sigma}}u_{(i)}(x-y^{i}_n)$$  must satisfy
\begin{eqnarray*}
v^{k+1}_n\rightarrow 0,\qquad \text{in}~\mathcal{E}_{p,q},
\end{eqnarray*}
and hence
\begin{eqnarray*}
c=\Phi(u_{(0)})+\underset {i=1}{\overset{k}{\Sigma}}\bar{\Phi}(u_{(i)}).
\end{eqnarray*}
Therefore, we complete the proof of Theorem \ref{t-0}.\qed

\

\section{Acknowledgements}
We would like to take this opportunity to express our gratitude to Genggeng Huang, who has given us so much valuable suggestions on our results, and has tried his best to improve our paper. B. H. is supported by NSFC, no.11831004.

\end{document}